\documentclass[11pt]{amsart}
  \usepackage{amsmath, amsfonts, amssymb,amsthm}

\newtheorem{theorem}{Theorem}
\newtheorem{lemma}{Lemma}
\newtheorem{cor}{Corollary}

\theoremstyle{definition}
\newtheorem{exam}[theorem]{Example}

\def\min{\mathop{\mathrm{min}}}

\newcommand{\C}{\mathbb{C}}

\newcommand{\N}{\mathbb{N}}

\newcommand{\PP}{\mathbb{P}}

\newcommand{\Z}{\mathbb{Z}}
\newcommand{\R}{\mathbb{R}}

\newcommand{\ax}{\rightarrow }

\begin{document}
\title{{On Cartan's second main theorems for holomorphic curves on $M-$ punctured complex planes }}

\author{Nguyen Van Thin}
\address{Department of Mathematics, Thai Nguyen University of Education, Luong Ngoc Quyen Street, Thai Nguyen city, Viet Nam.}
\email{thinmath@gmail.com}


\begin{abstract}

In this paper, we give some extension of fundamental theorems in Nevanlinna - Cartan theory for holomorphic curve on $M-$
 punctured complex planes. Detail, we prove some fundamental theorems for holomorphic curves on $M$ - punctured complex
 plane $\Omega=\mathbb C\setminus \{c_1, \dots, c_M\}$ intersecting a finite set of fixed hyperplanes in $\PP^n(\C)$  and fixed
 hypersurfaces in general position on complex 
projective variety with  the level of truncation, where $M\ge 2$ is an integer number and $c_1, \dots, c_M$ are distinct complex 
numbers. Note that in here, the holomorphic curves may be contained the essential singularity points on complex plane
 (at $c_j, j=1, \dots, M$)  which have not been considered before in other references according to my understanding. 
As an application, we establish a result for uniqueness problem of holomorphic curve by inverse image of a hypersurface, 
 it is improvement of some results  before \cite{DR, P} in this trend.
\end{abstract}
\baselineskip=16truept 
\maketitle 
\pagestyle{myheadings}
\markboth{}{}

\section{ Introduction and main results}

One of the main topic in studying the meromorphic functions is the Nevanlinna theory.  In 1933, H. Cartan \cite{CA} generalized the Nevanlinna theory for holomorphic curve in projective space which is now called the Nevanlinna - Cartan theory.
 Since that time, this problem has been studied intensively by many authors. Nevanlinna - Cartan theory has found various applications
in complex analysis and geometry complex such as uniqueness set theory, normal family theory, problem extension of holomorphic mapping and the property hyperbolic of algebraic variety.

In 2004, M. Ru \cite{Ru1} proved the second main theorem for hypersurface the counterpart of the result of Corvaja and Zannier
 \cite{PU} in approximation diophantine. In 2009, M. Ru \cite{Ru3} proved the second main theorem for holomorphic curves
from $\C$ into Algebraic variety on projective space. In 2010, Z. Chen, M. Ru and  Q. Yan \cite{Ru4} gave an improvement of Ru's result \cite{Ru3} with the level of truncation.  For meromorphic functions on Annuli in complex plane 
$\mathbb C$, in 2005, A. Y. Khrystiyanyn and A. A. Kondratyuk \cite{KK, KK1} showed some problems for the second main 
theorem and defect relation. In 2007, M. O. Hanyak and A. A. Kondratyuk \cite{KH} generalized the second main theorem of 
A. Y. Khrystiyanyn and A. A. Kondratyuk for $M-$ punctured complex planes. In 2015, H. T. Phuong and N. V. Thin \cite{TP}
have been generalized the results of  A. Y. Khrystiyanyn and A. A. Kondratyuk for holomorphic curves on annuli intersecting a 
finite set of fixed hyperplanes in general position in $\PP^n(\C)$ with ramification. In this paper, we will prove some fundamental theorems for holomorphic mappings from $M$ - punctured complex planes
  to $\mathbb P^n(\mathbb C)$ intersecting a finite set of hyperplanes which is an extension the result of Phuong and Thin \cite{TP}.
Futhermore, we extend the second main theorem of M. Ru \cite{Ru3} for holomorphic curves from $M$ - punctured complex 
planes into Algebraic variety in $\PP^n(\C)$ intersecting a finite set of hypersurfaces in general position.

First we remind some definitions in \cite{KH}. Let $c_j \in \C$, $j\in \{1, \dots, M\}$ be the set of distinct points, where $M\ge 
2$ is a positive integer. Then, $\Omega = \C \setminus \cup_{j=1}^{M}\{c_j\}$ is called the $M-$ punctured planes. Denotes $d=\dfrac{1}{2}\min\{|c_j-c_k|: j\ne k\}$
 and $r_0=1/d+\max\{|c_j|: 1\le j\le M\}.$ It is easy to see that $1/{r_0}<d, \overline D_{1/r_0}(c_j)\cap \overline D_{1/r_0}(c_k)=
\emptyset,$ 
$j\ne k,$ and $\overline D_{1/r_0}(c_j)\subset D_{r_0}(0)$, $j\in \{1, 2, \dots, M\},$ where $D_r(c)$ denotes a disk of radius $r>0$ 
centered at $c.$ For an arbitrary $t\ge r_0,$ we define
$$ \Omega_{t} = D_{t}(0) \setminus \cup_{j=1}^{M} \overline D_{1/t}(c_j).$$
Using the above notation, we conclude that $\Omega_{r_0} \subset \Omega_r, r_0<r \le +\infty.$

Let $ f=(f_{0}:\dots:f_{n}) : \Omega \ax \PP^n(\C)$  be a holomorphic map where $f_{0},\dots,f_{n}$ are holomorphic functions
 and  without common zeros in $\Omega$. For $r_0< r < +\infty$, the characteristic function $T_{f}(r)$ of $f$ is defined by 
$$T_{f}(r) =\dfrac{1}{2\pi}\int_{0}^{2\pi} \log \|f(re^{i\theta})\|d\theta  +\dfrac{1}{2\pi}\sum_{j=1}^{M}\int_{0}^{2\pi} 
\log \|f(c_j+\dfrac{1}{r}e^{i\theta})\|d\theta,$$
where $\|f(z)\|=\max\{ |f_0(z)|,\dots ,|f_n(z)| \}$. The above definition is independent, up to an additive constant, of the choice of the reduced representation of $f$. Futhermore, when holomorphic curve $f$ is holomorphic at $c_j, j=1, \dots, M,$ we have
$$T_{f}(r) =\dfrac{1}{2\pi}\int_{0}^{2\pi} \log \|f(re^{i\theta})\|d\theta +O(1).$$
Therefore, $T_f(r)$ is Nevanlinna-Cartan characteristic function  of holomorphic curve $f$ on $\mathbb C.$ Thus, our definition is an extension the definition of characteristic function for holomorphic curve on $\mathbb C$ to  $\Omega=\mathbb C\setminus \{c_1, \dots, c_M\}.$ We add the quantity
$$\dfrac{1}{2\pi}\sum_{j=1}^{M}\int_{0}^{2\pi} 
\log \|f(c_j+\dfrac{1}{r}e^{i\theta})\|d\theta$$
in the original definition of $T_f(r)$ to control the growth of f in the  neighborhood of the essential singularity points.

Let $H$ be a hyperplane in $\PP^n(\C)$ and 
$$L(z_0,\dots,z_n) = \sum_{j=0}^{n}a_{j}z_{j}$$
be linear form defined $H$, where $a_{j}\in \C, \ j=0,\dots,n,$ are constants. Denote $a=(a_{0},\dots,a_{n})$ by the non-zero associated vector with $H$. And denote 
$$(H,f)=(a,f) = \sum_{j=0}^{n}a_{j}f_{j}.$$
Under the assumption that $(a,f) \not\equiv\ 0$, for $r_0< r < +\infty$, the proximity function  of $f$ with respect to $H$ is defined as
$$m_{f}(r,H)=\dfrac{1}{2\pi}\int_{0}^{2\pi} \log \dfrac {\|f(re^{i\theta})\|}{|(a,f)(re^{i\theta})|} d\theta+
\dfrac{1}{2\pi}\sum_{j=1}^{M}\int_{0}^{2\pi} \log \dfrac {\|f(c_j+\dfrac{1}{r}e^{i\theta})\|}{|(a,f)(c_j+\dfrac{1}{r}e^{i\theta})|} d\theta,$$
where the above definition is independent, up to an additive constant, of the choice of the reduced representation of $f$.

We denote $n_{f}(r,H)$ by the number of zeros of $(a,f)$ in $\overline \Omega_r.$  The counting function of $f$ is defined by
$$N_{f}(r,H)=\int_{r_0}^r \dfrac{n_{f}(t,H)}{t} dt.$$
Now let $\delta$ be a positive integer, we denote $n_{f}^{\delta}(r,H)$ by the numbers of zeros of $(a,f)$ in $\overline \Omega_{r},$ 
where any zero of multiplicity greater than $\delta$ is ``truncated" and counted as if it only had multiplicity $\delta$.  The truncated 
counting function of $f$ is defined by
$$N^\delta_{f}(r,H)=\int_{r_0}^r \dfrac{n^\delta_{f}(t,H)}{t} dt.$$

Recall that hyperplanes $H_1,\dots,H_q, \ q > n$, in $\PP^n(\C)$ are said to be in general position if for any distinct $i_1,\dots,i_{n+1} \in \{1,\dots,q\}$,
$$\bigcap_{k=1}^{n+1}\text{supp}(H_{i_k}) = \emptyset,$$
this is equivalence to the $H_{i_1},\dots, H_{i_{n+1}}$ being linearly independent.

In the case of hypersurface, we may define the proximity function, counting functions of holomorphic curve $f$ similarly.  
Let $D$ be a hypersurface in $\PP^n(\C)$ of degree $d$. Let $Q$ be the homogeneous polynomial of degree $d$ defining $D$. 
Under the assumption that $ Q(f) \not\equiv 0.$ Then, the proximity function $m_f(r, D)$ of $f$ is defined by
$$m_{f}(r,D)=\dfrac{1}{2\pi}\int_{0}^{2\pi} \log \dfrac {\|f(re^{i\theta})\|^{d}}{|Q(f)(re^{i\theta})|} d\theta+
\dfrac{1}{2\pi}\sum_{j=1}^{M}\int_{0}^{2\pi} \log \dfrac {\|f(c_j+\dfrac{1}{r}e^{i\theta})\|^{d}}{|Q(f)(c_j+\dfrac{1}{r}e^{i\theta})|} d\theta,$$
where the above definition is independent, up to an additive constant, of the choice of the reduced representation of $f$.
The next, we denote $n_{f}(r,D)$ by the number of zeros of $Q(f)$ in  $\overline \Omega_r.$  The counting function 
$N_f(r, D)$  of  $f$ is defined by
$$N_{f}(r,D)=\int_{r_0}^r \dfrac{n_{f}(t,D)}{t} dt.$$
Now let $\delta$ be a positive integer, we denote $n_{f}^{\delta}(r,D)$ by the numbers of zeros of $Q(f)$ in $\overline \Omega_{r},$ 
where any zero of multiplicity greater than $\delta$ is ``truncated" and counted as if it only had multiplicity $\delta$.  The truncated 
counting function of $f$ is defined by
$$N^\delta_{f}(r,D)=\int_{r_0}^r \dfrac{n^\delta_{f}(t,D)}{t} dt.$$

Let $V\subset \PP^N(\C)$ be a smooth complex projective variety of dimension $n\ge 1.$ Let $D_1, \dots, D_q$ be hypersurfaces
in $\PP^N(\C),$ where $q>n.$ The hypersurfaces $D_1, \dots, D_q$ are said to be {\it in general position on V}\; if for every subset
$\{i_0, \dots, i_n\} \subset \{1, \dots, q\},$ we have
$$ V\cap \text{Supp}D_{i_{0}} \cap \dots \cap \text{Supp}D_{i_n}= \emptyset,$$
where $\text{Supp}(D)$ means the support of the divisor $D.$ A map $f: \Omega \to V$ is said to be {\it algebraically nondegenerate}
 if the image of $f$ is not contained in any proper subvarieties of $V.$

In this paper,  a notation $``\|"$ in the inequality  means that the inequality holds for $r\in(r_0,+\infty)$ outside a set  of finite measure.

Our main results are
\begin{theorem}\label{Th1}
 Let $D$ be a hypersurfaces in $\PP^n(\C)$ defining the homogeneous polynomial $Q$ with degree 
$d$ and $f=(f_{0}:\dots:f_{n}) : \Omega \ax \PP^{n}(\C)$ be a holomorphic curve whose image is not contained $D$. Then we have for any $r_0< r <+\infty$,
$$dT_{f}(r)=m_{f}(r,D)+N_{f}(r,D)+O(1),$$
where $O(1)$ is a constant independent of $r$.
\end{theorem}

Note that, when $d=1$, we get the corollary as following:
\begin{cor}\label{cor1}
  Let $H$ be a hyperplane in $\PP^n(\C)$ and $f=(f_{0}:\dots:f_{n}) : \Omega \ax \PP^{n}(\C)$  be a 
holomorphic curve whose image is not contained $H$. Then we have for any $r_0< r <+\infty$,
$$T_{f}(r)=m_{f}(r,H)+N_{f}(r,H)+O(1),$$
where $O(1)$ is a constant independent of $r$.
\end{cor}

\begin{theorem}\label{Th2}
Let $ f=(f_{0}:\dots:f_{n}) : \Omega \ax \PP^{n}(\C)$  be a linearly non-degenerate holomorphic
 curve  and $H_{1},\dots , H_{q}$ be hyperplanes in $\PP^{n}(\C)$ in general position.  Then we have for any $r_0< r <+\infty,$
$$\| \quad (q-n-1)T_{f}(r)\leqslant\sum_{l=1}^{q} N^{n}_{f}(r,H_l)+O(\log r+ \log T_{f}(r)).$$
\end{theorem}

\begin{theorem}\label{Th3}
 Let $V\subset \PP^{N}(\C)$ be a complex projective variety of dimension $n\ge 1.$ Let $D_1, \dots, D_q$
 be hypersurfaces in $\PP^N(\C)$ of degree $d_j,$ located in general position on $V.$ Let $d$ be the least common multiple of the 
$d_i,$ $i=1, \dots, q.$ Let $ f=(f_{0}:\dots:f_{N}) : \Omega \ax V$  be an algebraically non-degenerate holomorphic map. Let $\varepsilon >0$
 and $$ \alpha \ge \dfrac{n^nd^{n^2+n}(19nI(\varepsilon^{-1}))^n(\deg V)^{n+1}}{n!},$$
where $I(x):=\min\{k\in \N: k>x\}$ for a positive real number $x.$ Then
$$\| \quad (q(1-\varepsilon/3)-(n+1)-\varepsilon/3)T_{f}(r)\leqslant\sum_{l=1}^{q}d_l^{-1} N^{\alpha}_{f}(r,Q_l)
+O(\log r+ \log T_{f}(r)).$$
\end{theorem}

When $f$ is holomorphic at $c_j, j=1, \dots, M,$ we get following result of Ru et. al \cite{Ru4}. Thus Theorem \ref{Th3} is an extension of Ru et. al \cite{Ru4} and Ru \cite{Ru3} for holomorphic curve with essential singularity points.

\begin{theorem}\label{Th4}
 Let $V\subset \PP^{N}(\C)$ be a complex projective variety of dimension $n\ge 1.$ Let $D_1, \dots, D_q$
 be hypersurfaces in $\PP^N(\C)$ of degree $d_j,$ located in general position on $V.$ Let $d$ be the least common multiple of the 
$d_i,$ $i=1, \dots, q.$ Let $ f=(f_{0}:\dots:f_{N}) : \C \ax V$  be an algebraically non-degenerate holomorphic map. Let $\varepsilon >0$
 and $$ \alpha \ge \dfrac{n^nd^{n^2+n}(19nI(\varepsilon^{-1}))^n(\deg V)^{n+1}}{n!},$$
where $I(x):=\min\{k\in \N: k>x\}$ for a positive real number $x.$ Then
$$\quad (q(1-\varepsilon/3)-(n+1)-\varepsilon/3)T_{f}(r)\leqslant\sum_{l=1}^{q}d_l^{-1} N^{\alpha}_{f}(r,Q_l)+O(\log r+ \log T_{f}(r)),$$
holds for all $r\in (0, +\infty)$ outside a set of finite measure.
\end{theorem}

Theorem \ref{Th1} and Corollary \ref{cor1} is first main theorem, and Theorem \ref{Th2} is second main theorem for holomorphic curves from $M-$ 
punctured $\Omega$ to $\PP^n(\C)$ intersecting a collection of fixed hyperplanes in general position with truncated counting functions. 
Theorem \ref{Th3} is second main theorem for holomorphic curves from $M-$ punctured
 $\Omega$ to $V$ intersecting a collection of  fixed hypersurfaces in general position with  the level of 
truncation. When one applies inequalities of second main theorem type, it is often crucial to the application to have the inequality 
with truncated counting functions. For example, all existing constructions of unique range sets depend on a second main theorem  
with truncated counting functions. 

\begin{theorem}\label{Th10} Let $f: \Omega \to \mathbb P^N(\mathbb C)$ be an algebraically nondegenerate
holomorphic curve. Let $d $ and $n$ be two integers with $n>N(d+N+1).$ Let 
$\mathcal H_i=\{z\in \mathbb P^{N}(\mathbb C), \mathcal H_i(z)=0\}, 0\le i\le N,$ be hyperplanes in $\mathbb P^N(\mathbb C).$ Let $D_i=\{z\in \mathbb P^{N}(\mathbb C), Q_i(z)=0\}, 
0\le i\le N,$ be hypersurfaces of degree $d$ such that the hypersurfaces
 $\{\mathcal H_0^nQ_0=0\}, \dots, \{\mathcal H_N^nQ_N=0\}$ are in general position in 
$\mathbb P^N(\mathbb C).$ Let 
 $D=\{z\in \mathbb P^N(\mathbb C), \sum_{i=0}^{N}\mathcal H_i^nQ_i=0\}.$ Then
\begin{align*}
\|(n-(d+N+1)N) T_{f}(r)&+\sum_{i=0}^{N}( N_{f}(r, D_i)
- N_{f}^{N}(r, D_i))\\
&\le N_{f}^{N}(r, D)+o(T_{f}(r)).
\end{align*}
\end{theorem}

We give a hypersurfaces satisfying Theorem \ref{Th10}.
\begin{exam}\label{exam2} Let $D_i=\{z=(x_0:\dots:x_N)\in \mathbb P^{N}(\mathbb C),  x_i^d=0\}, 
0\le i\le N,$ be hypersurfaces of degree $d.$  Let $\mathcal H_i=\{z=(x_0:\dots:x_N)\in \mathbb P^{N}(\mathbb C), 
\sum_{t=0}^{i}x_t=0\}.$
We see that the hypersurfaces
 $\{(\sum_{t=0}^{i}x_t)^{n}x_i^{d}=0\}, 0\le i\le N,$ are in general position in $\mathbb P^N(\mathbb C).$ Then 
 $$D=\{z\in \mathbb P^N(\mathbb C), \sum_{i=0}^{N}(\sum_{t=0}^{i}x_t)^{n}x_i^{d}=0\}$$ satisfies the 
Theorem \ref{Th10} with $n>N(d+N+1).$
\end{exam}

As an application of Theorem \ref{Th10}, we prove the uniqueness theorem for holomorphic curves on $M-$ punctured 
complex plane by inverse image of a Fermat hypersurface.
\begin{theorem}\label{Th6}
Let $f, g :  \Omega \to \mathbb P^N(\mathbb C)$ be two algebraically nondegenerate
holomorphic curves, and $n$ be a integer with $n>N(d+N+3).$ Let $D$ be a hypersurface as in Theorem \ref{Th10}.  
Suppose that $f(z)=g(z)$ on $f^{-1}(D)\cup g^{-1}(D),$ then $f\equiv g.$
\end{theorem}

When $f$ and $g$ are holomorphic at $c_j, j=1, \dots, M,$ we obtain the corollary as following:

\begin{cor}\label{corth6}
Let $f, g : \mathbb C \to \mathbb P^N(\mathbb C)$ be two algebraically nondegenerate
holomorphic curves, and $n$ be a integer with $n>N(d+N+3).$ Let $D$ be a hypersurface as in Theorem \ref{Th10}.  
Suppose that $f(z)=g(z)$ on $f^{-1}(D)\cup g^{-1}(D),$ then $f\equiv g.$
\end{cor}

We reduce the number hypersurfaces in before results. The authors \cite{DT, DR, P} studied the uniqueness problem 
with a number lager hypersurfaces. Here, we only need a hypersurface.


\section{Some preliminaries in Nevanlinna theory for meromorphic functions}
\def\theequation{2.\arabic{equation}}
\setcounter{equation}{0} 
In order to prove  theorems, we need the following lemmas. Let $f$ be a meromorphic function on $M-$ punctured $\Omega$ and 
$r \in (r_0, +\infty)$, we recall that
\begin{align*}
m_0(r, f)&=\dfrac{1}{2\pi}\int_{0}^{2\pi}\log^{+}|f(re^{i\theta})|d\theta+\dfrac{1}{2\pi}\sum_{j=1}^{M}\int_{0}^{2\pi}
\log^{+}|f(c_j+\dfrac{1}{r}e^{i\theta})|d\theta\\
&-\dfrac{1}{2\pi}\int_{0}^{2\pi}\log^{+}|f(r_0e^{i\theta})|-\dfrac{1}{2\pi}\sum_{j=1}^{M}\int_{0}^{2\pi}
\log^{+}|f(c_j+\dfrac{1}{r_0}e^{i\theta})|d\theta.
\end{align*}
We denote $n_0(r, f)$ by the numbers of its poles in $\overline \Omega_r.$ The counting function $N_0(r, f)$ of $f$ is defined by
$$ N_0(r, f)=\dfrac{1}{2\pi}\int_{r_0}^{r}\dfrac{n_0(t, f)}{t}dt. $$
The function
$$ T_0(r, f)=N_0(r, f)+m_0(r, f)$$
is called {\it the Nevanlinna characteristic of $f$.}

\begin{lemma}\label{lm21}\cite{KH} (Jensen's Theorem for $M-$ punctured planes)
Let $f$ be a  non-constant meromorphic 
function in an $M-$ punctured plane $\Omega$ not identically equal to zero and let $r_0<r<+\infty$. Then 

\begin{align*}
N_0\bigg(r,\dfrac{1}{f}\bigg) - N_0(r,f)& = \dfrac{1}{2\pi} \int_0^{2\pi} \log |f(re^{i\theta})| d\theta + 
\dfrac{1}{2\pi}\sum_{j=1}^{M} \int_0^{2\pi} \log |f(c_j+\dfrac{1}{r}e^{i\theta})| d\theta \\
&- \dfrac{1}{2\pi} \int_0^{2\pi} \log |f(r_0e^{i\theta})| d\theta-
\dfrac{1}{2\pi}\sum_{j=1}^{M} \int_0^{2\pi} \log |f(c_j+\dfrac{1}{r_0}e^{i\theta})| d\theta.
\end{align*}
\end{lemma}

\begin{lemma}\label{lm22}\cite{KH} Let $f$ be a  non-constant meromorphic function on $\Omega.$ Then, we have the 
equality 
$$\| \quad m_{0}(r,\dfrac{f'}{f})=O(\log r+\log T_{0}(r,f)),$$ 
holds for all $r \in (r_0, +\infty)$ outside a set of finite measure.
\end{lemma}

Let $X\subset \PP^N(\C)$ be a projective variety of dimension $n$ and degree $\Delta.$  Let $I_X$ be the prime ideal in 
$\C[x_0, \dots, x_N]$ defining $X$. Denote by $\C[x_0, \dots, x_N]_m$ the vector space of homogeneous
polynomials in $\C[x_0, \dots, x_N]$ of degree $m$ (including 0). Put $I_X(m):=\C[x_0, \dots, x_N]_m\cap I_X.$ The Hilbert
 function $H_X$ of $X$ is defined by
$$ H_X(m):=\dim \C[x_0, \dots, x_N]_m / I_X(m).$$
For each tuple $c=(c_0, \dots, c_N) \in \R^{N+1}_{\ge 0}$ and  $m\in \N$, we define the $m$-th Hilbert weight
$S_X(m, c)$ of $X$ with respect to $c$ by
$$ S_X(m, c):=\max \sum_{i=1}^{H_X(m)}I_i.c,$$
where $I_i=(I_{i0}, \dots, I_{iN}) \in  \N_{0}^{N+1}$ and the maximum is taken over all sets $\{x^{I_i}=x_0^{T_{i0}}\dots x_N^{I_{iN}}\}$
whose residue classes modulo $I_X(m)$ form a basis of the vector space  $C[x_0, \dots, x_N]_m/I_X(m).$

\begin{lemma}\label{lm23}\cite{Ru3}  
Let $X\subset \PP^{N}(\C)$ be an algebraic variety of dimension $n$ and degree $\Delta$. Let $m>\Delta$ be an integer and let 
$c=(c_0, \dots, c_N) \in \R_{\ge 0}^{N+1}.$  Then
$$ \dfrac{1}{mH_X(m)}S_X(m, c)\ge \dfrac{1}{(n+1)\Delta}e_X(c)-\dfrac{(2n+1)\Delta}{m}\max_{0\le i\le N}c_i.$$
\end{lemma}

\begin{lemma}\label{lm24}\cite{Ru3} Let $Y$ be a subvariety of $\PP^{q-1}(\C)$ of dimension $n$ and degree $\Delta.$ Let
$c=(c_1, \dots, c_q)$ be a tupe of positive reals. Let $\{i_0, \dots, i_n\}$ be a subset of
$\{1, \dots , q \}$ such that $\{ y_{i_0}=\dots=y_{i_n}=0\}\cap Y = \emptyset$.  Then
$$ e_Y(c) \ge (c_{i_0}+\dots c_{i_n})\Delta. $$
\end{lemma}

\section{Proof of Theorems}
\def\theequation{3.\arabic{equation}}
\setcounter{equation}{0}

\begin{proof}[Proof of Theorem 1]
The first, we note that $N_0(r, Q(f))=0.$ By the definitions of $T_{f}(r)$, $N_{f}(r,H)$, $ m_{f}(r,H)$ and apply to Lemma \ref{lm21}
for $Q(f) \not\equiv 0$, we have
\begin{align*}
N_{f}(r, D)&=N_{0}(r, \dfrac{1}{Q(f)})\\
&= \dfrac{1}{2\pi} \int_0^{2\pi} \log |Q(f)(re^{i\theta})| d\theta + 
\dfrac{1}{2\pi}\sum_{j=1}^{M} \int_0^{2\pi} \log |Q(f)(c_j+\dfrac{1}{r}e^{i\theta})| d\theta +O(1).
\end{align*}
Hence, we get
\begin{align*}
N_{f}(r,D)&+m_{f}(r,D) \\
&=\dfrac{1}{2\pi}\int_{0}^{2\pi} \log \dfrac{\|f(re^{i\theta})\|^{d}}{|Q(f)(re^{i\theta})|}d\theta+ 
\dfrac{1}{2\pi}\sum_{j=1}^{M}\int_{0}^{2\pi} \log \dfrac{\|f(c_j+\dfrac{1}{r}e^{i\theta})\|^{d}}
{|Q(f)(c_j+\dfrac{1}{r}e^{i\theta})|}d\theta \\
&\quad+\dfrac{1}{2\pi}\int_{0}^{2\pi} \log |Q(f)(re^{i\theta})| d\theta +\dfrac{1}{2\pi}\sum_{j=1}^{M}\int_{0}^{2\pi}
 \log |Q(f)(c_j+\dfrac{1}{r}e^{i\theta})|d\theta +O(1)\\
&= \dfrac{1}{2\pi}\int_{0}^{2\pi} \log \|f(re^{i\theta})\|^{d}d\theta+ \dfrac{1}{2\pi}\sum_{j=1}^{M}\int_{0}^{2\pi} 
\log \|f(c_j+\dfrac{1}{r}e^{i\theta})\|^{d} d\theta +O(1) \\
&=  dT_f(r)+O(1).
\end{align*}
This is conclusion of Theorem \ref{Th1}.
\end{proof}

\begin{proof}[Proof of Theorem \ref{Th2}]
To prove the Theorem \ref{Th2}, we need some lemmas. First we recall the property of Wronskian. Let $f=(f_0:\dots:f_n): \Omega \ax \PP^n(\C)$ be holomorphic curves, the determinant of Wronskian of $f$ is defined by 
\[ W = W(f) = W(f_0,\dots,f_n) =
\left|\begin{array}{cccc}
f_{0}(z)&f_{1}(z)&\cdot & f_{n}(z)\\
f'_{0}(z)&f'_{1}(z)&\cdot & f'_{n}(z)\\
\vdots&\vdots&\ddots&\vdots\\
f^{(n)}_{0}(z)&f^{(n)}_{1}(z)&\cdot & f^{(n)}_{n}(z)\\
\end{array}
\right|.
\]
We denote by $N_W(r,0)$ the counting function of zeros of $W(f_0,\dots,f_n)$ in $\overline \Omega_r$, namely
$$N_W(r,0) = N_0(r,\dfrac{1}{W})+O(1).$$

Let $L_0,\dots,L_{n}$ are linearly independent forms  of $z_0,\dots,z_n$. For $j=0,\dots,n$, set $$F_j(z):= L_j(f(z)).$$ By the 
property of Wronskian there exists a constant $C \ne 0$ such that
$$|W(F_0,\dots,F_n)| =C|W(f_{0},\dots ,f_{n})|.$$

\begin{lemma}\label{lm31}
Let $ f=(f_{0}:\dots:f_{n}) : \Omega \ax \PP^{n}(\C)$  be a linearly non-degenerate holomorphic 
curve  and $H_{1},\dots , H_{q}$ be hyperplanes in $\PP^{n}(\C)$ in general position. Then we have the inequality
\begin{align*}
\| \quad \int_{0}^{2\pi}\max_{K}\sum_{l \in K}\log \dfrac {\|f(re^{i\theta})\|}{|(a_l,f)(re^{i\theta})|}& 
\dfrac{d\theta}{2\pi}
+\sum_{j=1}^{M}\int_{0}^{2\pi}\max_{K}\sum_{l \in K}\log \dfrac {\|f(c_j+\dfrac{1}{r}e^{i\theta})\|}{|(a_l,f)(c_j+
\dfrac{1}{r}e^{i\theta})|} \dfrac{d\theta}{2\pi} \\
&\leqslant (n+1)T_{f}(r) - N_{W}(r,0) +O(\log r+ \log T_{f}(r)).
\end{align*}
Here the maximum is taken over all subsets $K$ of $\{1,\dots,q\}$ such that $a_{l}$, $l\in K$, are linearly independent.
\end{lemma}

First, we prove 
\begin{align}\label{ct3.1}
\| \quad \int_{0}^{2\pi}\max_{K}\sum_{l\in K}&\log \dfrac {\|f(re^{i\theta})\|}{|(a_j,f)(re^{i\theta})|} \dfrac{d\theta}{2\pi}+\dfrac{1}{2\pi} \int_{0}^{2\pi}\log |W(f)(re^{i\theta})| d\theta\\
& \leqslant (n+1)\dfrac{1}{2\pi}\int_{0}^{2\pi} \log \|f(re^{i\theta})\| d\theta +O(\log r+ \log T_{f}(r)), \notag
\end{align}
holds for any $r \in (r_0, +\infty)$. Let $K \subset \{1,\dots,q\}$ such that $a_l, l \in K,$ are linearly independent. Without loss of 
generality, we may assume that $ q\geqslant n+1$ and $\# K=n+1.$ Let $\mathcal T$ is the set of all injective maps 
$\mu:\{ 0,1,\dots,n\} \ax \{1,\dots,q\}$. Then we have
\begin{align*}
\int_{0}^{2\pi}\max_{K}\sum_{l \in K}&\log \dfrac {\|f(re^{i\theta})\|}{|(a_j,f)(re^{i\theta})|} \dfrac{d\theta}{2\pi} \\
&= \int_{0}^{2\pi}\max_{\mu \in \mathcal  T}\sum_{l=0}^n\log \dfrac {\|f(re^{i\theta})\|}{|(a_{\mu(l)},f)(re^{i\theta})|} 
\dfrac{d\theta}{2\pi}\\
&= \int_{0}^{2\pi}\log \bigg\{\max_{\mu \in \mathcal  T} \dfrac {\|f(re^{i\theta})\|^{n+1}}{\prod\limits_{l=0}^n|(a_{\mu(l)},f)
(re^{i\theta})|}\bigg\} \dfrac{d\theta}{2\pi} +O(1)\\
&\leqslant \int_{0}^{2\pi}\log \sum_{\mu \in  \mathcal  T}\dfrac {\|f(re^{i\theta})\|^{n+1}}{\prod\limits_{l=0}^n|(a_{\mu(l)},f)
(re^{i\theta})|} \dfrac{d\theta}{2\pi}+O(1)\\
&= \int_{0}^{2\pi}\log \sum_{\mu \in \mathcal  T}\dfrac {|W((a_{\mu(0)},f),\dots,(a_{\mu(n)},f))(re^{i\theta})|}
{\prod\limits_{l=0}^n|(a_{\mu(l)},f)(re^{i\theta})|} \dfrac{d\theta}{2\pi}\\
& \quad+\int_{0}^{2\pi}\log \sum_{\mu \in \mathcal  T}\dfrac{\|f(re^{i\theta})\|^{n+1}} {|W((a_{\mu(0)},f),
\dots,(a_{\mu(n)},f))(re^{i\theta})|} \dfrac{d\theta}{2\pi}+O(1).
\end{align*}
By the property of Wronskian, we see that
$$|W((a_{\mu(0)},f),\dots,(a_{\mu(n)},f))| =C|W(f_{0},\dots ,f_{n})|,$$
where $C \ne 0$ is a constant. So we obtain
\begin{align}\label{ct3.2}
\int_{0}^{2\pi}&\max_{K}\sum_{l \in K}\log \dfrac {\|f(re^{i\theta})\|}{|(a_l,f)(re^{i\theta})|} \dfrac{d\theta}{2\pi} \\
&\leqslant \int_{0}^{2\pi}\log \sum_{\mu \in T}\dfrac {|W((a_{\mu(0)},f),\dots,(a_{\mu(n)},f))(re^{i\theta})|}
{\prod\limits_{l=0}^n|(a_{\mu(l)},f)(re^{i\theta})|} \dfrac{d\theta}{2\pi}\notag \\
& \qquad+\int_{0}^{2\pi}\log \dfrac{\|f(re^{i\theta})\|^{n+1}} {|W(f_0,\dots,f_n)(re^{i\theta})|} \dfrac{d\theta}{2\pi}+O(1).\notag
\end{align}
Take
$$ g_{\mu(j)}= \dfrac{(a_{\mu(j)},f)}{(a_{\mu(0)},f)}, j=1, \dots, n.$$
From property of Wronskian (see \cite{IL}, Proposition 1.4.3), we have
\begin{align}\label{ct3a2}
\dfrac {W((a_{\mu(0)},f),\dots,(a_{\mu(n)},f))}{\prod\limits_{j=0}^n(a_{\mu(j)},f)} 
&= \dfrac{W(1, \dfrac{(a_{\mu(1)},f)}{(a_{\mu(0)},f)}, \dots, \dfrac{(a_{\mu(n)},f)}{(a_{\mu(0)},f)})}
{\dfrac{(a_{\mu(1)},f)}{(a_{\mu(0)},f)}\dots \dfrac{(a_{\mu(n)},f)}{(a_{\mu(0)},f)} }\notag\\
&=\left|\begin{array}{cccc}
1& 1&\dots & 1\\
0& \dfrac{g'_{\mu(1)}}{ g_{\mu(1)}}&\dots & \dfrac{g'_{\mu(n)}}{ g_{\mu(n)}}\\
\vdots&\vdots&\ddots&\vdots\\
0 & \dfrac{g^{(n)}_{\mu(1)}}{ g_{\mu(1)}}&\dots & \dfrac{g^{(n)}_{\mu(n)}}{ g_{\mu(n)}}
\end{array}
\right|.
\end{align}

We see
\begin{align}\label{ct3.3}
\| \quad m\bigg(r,\dfrac{g_{\mu(j)}^{(k)}}{g_{\mu(j)}}\bigg)&\leqslant 
m_0\bigg(r,\dfrac{g_{\mu(j)}^{(k)}}{g_{\mu(j)}}\bigg) =m_0\bigg(r,\dfrac{g_{\mu(j)}^{(k)}}{g_{\mu(j)}^{(k-1)}}\dfrac{g_{\mu(j)}^{(k-1)}}
{g_{\mu(j)}^{(k-2)}}
\dots \dfrac{g_{\mu(j)}'}{g_{\mu(j)}}\bigg)\notag\\
&\le \sum_{\nu=1}^{k}m_0\bigg(r,\dfrac{g_{\mu(j)}^{(\nu)}}{g_{\mu(j)}^{(\nu-1)}}\bigg).
\end{align}

By Lemma \ref{lm22}, we have
\begin{align}\label{ct3.3a}
 m_0(r, \dfrac{g_{\mu(j)}^{'}}{g_{\mu(j)}}) =O(\log r+\log T_0(r,g_{\mu(j)}).
\end{align}
From the definition of $T_0(r, g_{\mu(j)}')$, $N_0(r, g_{\mu(j)}')$ and (\ref{ct3.3a}), we obtain
\begin{align}\label{ct3.4a}
T_0(r, g_{\mu(j)}')&=m_0(r,g_{\mu(j)}' )+N_0(r, g_{\mu(j)}')\\
&=m_0(r,\dfrac{g_{\mu(j)}'}{g_{\mu(j)}}.g_{\mu(j)})+N_0(r, g_{\mu(j)}')\notag\\
&\le m_0(r, g_{\mu(j)})+N_0(r, g_{\mu(j)})+\overline N_0(r, g_{\mu(j)})+O(\log r+\log T_0(r,g_{\mu(j)})\notag\\
&=2T_0(r, g_{\mu(j)})+O(\log r+\log T_0(r,g_{\mu(j)}).\notag
\end{align}
Similarly, again using Lemma \ref{lm22} and (\ref{ct3.4a}), we have
\begin{align}\label{ct3.4b}
T_0(r, g_{\mu(j)}'')&=m_0(r,g_{\mu(j)}'')+N_0(r,g_{\mu(j)}'')\\
&=m_0(r,\dfrac{g_{\mu(j)}''}{g_{\mu(j)}'}.g_{\mu(j)}')+N_0(r,g_{\mu(j)}'')\notag\\
&\le m_0(r, g_{\mu(j)}')+N_0(r,g_{\mu(j)})+2\overline N_0(r,g_{\mu(j)})+O(\log r+\log T_0(r, g_{\mu(j)}')\notag\\
&=3T_0(r, g_{\mu(j)})+O(\log r+\log T_0(r,g_{\mu(j)}).\notag
\end{align}
By argument as (\ref{ct3.4b}) and using inductive method,  we obtain that the inequality
\begin{align}\label{ct3.4c}
T_0(r, g_{\mu(j)}^{(\nu)})\le (\nu+1)T_0(r, g_{\mu(j)})+O(\log r+\log T_0(r,g_{\mu(j)})
\end{align}
holds for all $\nu \in \N^{*}.$ Furthemore, by Lemma \ref{lm22}, we also have the equality
\begin{align}\label{ct3.4d}
m_0(r, \dfrac{g_{\mu(j)}^{(\nu+1)}}{g_{\mu(j)}^{(\nu)}})=O(\log r + \log T_0(r, g_{\mu(j)}^{(\nu)})),
\end{align}
holds for all $\nu \in \N.$
Combining (\ref{ct3.3}), (\ref{ct3.4c}) and (\ref{ct3.4d}), we get for any $k\in \{1,\dots,n\}$ and $j \in \{1,\dots,n\}$,

\begin{align}\label{ct3.3aa}
\| \quad m\bigg(r,\dfrac{g_{\mu(j)}^{(k)}}{g_{\mu(j)}}\bigg)&\leqslant O(\log r + \log T_0(r, g_{\mu(j)}).
\end{align}
By the definition of $T_{0}(r,g_{\mu(j)}),T_{f}(r)$, we have
\begin{align}\label{ctm1}
T_{0}(r,g_{\mu(j)}) &=m_{0}(r,g_{\mu(j)})+N_0(r, g_{\mu(j)})\notag\\
&=
\dfrac{1}{2\pi}\int_{0}^{2\pi} \log^{+} \Big| \dfrac{(a_{\mu(j)},f)}{(a_{\mu(0)},f)}\Big| d\theta  +
\dfrac{1}{2\pi}\sum_{j=1}^{m}
\int_{0}^{2\pi} \log^{+} \Big|\dfrac{(a_{\mu(j)},f)}{(a_{\mu(0)},f)}(c_j+\dfrac{1}{r}e^{i\theta})\Big| d\theta\notag\\
&\quad+\dfrac{1}{2\pi}\int_{0}^{2\pi} \log | (a_{\mu(0)},f)| d\theta  +
\dfrac{1}{2\pi}\sum_{j=1}^{m}
\int_{0}^{2\pi} \log |(a_{\mu(0)},f)(c_j+\dfrac{1}{r}e^{i\theta})| d\theta+O(1)\notag\\
&\le \dfrac{1}{2\pi}\int_{0}^{2\pi} \log \Big| \dfrac{(a_{\mu(j)},f)+(a_{\mu(0)},f)}{(a_{\mu(0)},f)}\Big| d\theta\notag\\
&\quad +\dfrac{1}{2\pi}\sum_{j=1}^{m}
\int_{0}^{2\pi} \log  \Big|\dfrac{(a_{\mu(j)},f)+(a_{\mu(0)},f)}{(a_{\mu(0)},f)}(c_j+\dfrac{1}{r}e^{i\theta})\Big| d\theta\notag\\
&\quad+\dfrac{1}{2\pi}\int_{0}^{2\pi} \log | (a_{\mu(0)},f)| d\theta  +
\dfrac{1}{2\pi}\sum_{j=1}^{m}
\int_{0}^{2\pi} \log |(a_{\mu(0)},f)(c_j+\dfrac{1}{r}e^{i\theta})| d\theta+O(1)\notag\\
&\le T_f(r)+O(1).
\end{align}

Hence for any $\mu \in \mathcal  T$, from (\ref{ct3a2}), (\ref{ct3.3aa}) and (\ref{ctm1}), we have 
\begin{align*}
\| \quad \int_{0}^{2\pi}\log^+\dfrac {|W((a_{\mu(0)},f),\dots,(a_{\mu(n)},f))(re^{i\theta})|}{\prod\limits_{l=0}^n|(a_{\mu(l)},f)
(re^{i\theta})|}\dfrac{d\theta}{2\pi} \leqslant O(\log r+\log T_f(r)).
\end{align*}
This implies
\begin{align}\label{ct3.4}
\| \quad \int_{0}^{2\pi} \log &\sum_{\mu \in \mathcal  T} \dfrac {|W((a_{\mu(0)},f),\dots,(a_{\mu(n)},f))
(re^{i\theta})|}{\prod\limits_{l=0}^n|(a_{\mu(l)},f)(re^{i\theta})|}\dfrac{d\theta}{2\pi} \\
&\leqslant \int_{0}^{2\pi} \log^+\sum_{\mu \in \mathcal  T} \dfrac {|W((a_{\mu(0)},f),\dots,(a_{\mu(n)},f))(re^{i\theta})|}
{\prod\limits_{l=0}^n|(a_{\mu(l)},f)(re^{i\theta})|}\dfrac{d\theta}{2\pi} \notag\\
&\leqslant  \sum_{\mu \in \mathcal  T} \int_{0}^{2\pi}  \log^+\dfrac {|W((a_{\mu(0)},f),\dots,(a_{\mu(n)},f))
(re^{i\theta})|}{\prod\limits_{l=0}^n|(a_{\mu(l)},f)(re^{i\theta})|}\dfrac{d\theta}{2\pi} +O(1) \notag\\
& \leqslant O(\log r+\log T_f(r)).\notag
\end{align}
We may obtain the inequality (\ref{ct3.1}) from (\ref{ct3.2}) and (\ref{ct3.4}). Similarly, we get
\begin{align}\label{ct3.5}
\| \quad \int_{0}^{2\pi}\max_{K}&\sum_{l \in K}\log \dfrac {\|f(c_j+\dfrac{1}{r}e^{i\theta})\|}
{|(a_l,f)(c_j+\dfrac{1}{r}e^{i\theta})|} \dfrac{d\theta}{2\pi} 
+ \dfrac{1}{2\pi}\int_{0}^{2\pi}\log |W(f)(c_j+\dfrac{1}{r}e^{i\theta})| d\theta\\
&\leqslant (n+1)\dfrac{1}{2\pi}\int_{0}^{2\pi} \log \| f(c_j+\dfrac{1}{r}e^{i\theta})\| d\theta+O(\log r+\log T_f(r))\notag
\end{align}
holds for any $r \in (r_0, +\infty)$ and for all $j=1, \dots, M$. Combining (\ref{ct3.1}) and (\ref{ct3.5}) we obtain
\begin{align*}
\| \quad \int_{0}^{2\pi}\max_{K}&\sum_{l \in K}\log \dfrac {\|f(re^{i\theta})\|}{|(a_l,f)(re^{i\theta})|} \dfrac{d\theta}{2\pi}
+\sum_{j=1}^{M} \int_{0}^{2\pi}\max_{K}\sum_{l \in K}\log \dfrac {\|f(c_j+\dfrac{1}{r}e^{i\theta})\|}{|(a_l,f)(c_j+\dfrac{1}{r}
e^{i\theta})|} \dfrac{d\theta}{2\pi} \\
&\leqslant (n+1)\bigg(\dfrac{1}{2\pi}\int_{0}^{2\pi} \log \| f(re^{i\theta})\| d\theta + \dfrac{1}{2\pi}\sum_{j=1}^{M}\int_{0}^{2\pi} 
\log \| f(c_j+\dfrac{1}{r}e^{i\theta})\| d\theta\bigg)\\
&- \dfrac{1}{2\pi}\bigg(\int_{0}^{2\pi}\log |W(f)(re^{i\theta})| d\theta+\sum_{j=1}^{M}\int_{0}^{2\pi}\log |W(f)(c_j+\dfrac{1}{r}
e^{i\theta})| d\theta\bigg)\\
&+O(\log r+\log T_f(r)).
\end{align*}
Since 
$$N_{W}(r,0)=\dfrac{1}{2\pi}\int_{0}^{2\pi}\log |W(f)(re^{i\theta})| d\theta+\dfrac{1}{2\pi}\sum_{j=1}^{M}\int_{0}^{2\pi}
\log |W(f)(c_j+\dfrac{1}{r}e^{i\theta})| d\theta+O(1),$$
we have the conclusion of this lemma.
\end{proof}

\begin{lemma}\label{lm32} Let $ f=(f_{0}:\dots:f_{n}) : \Omega \longrightarrow \PP^{n}(\C)$  be a linearly non-degenerate holomorphic curve  and $H_{1},\dots ,H_{q}$ be hyperplanes in $\PP^{n}(\C)$ in general position. Let $a_j$ be the vector associated with $H_j$ for $j=1,\dots,q$. Then
\begin{align*}
\sum_{l=1}^{q}m_{f}(r,H_{l}) \leqslant &\int_{0}^{2\pi}\max_{K}\sum_{l \in K}\log \dfrac {\|f(re^{i\theta})\|}{|(a_l,f)(re^{i\theta})|} \dfrac{d\theta}{2\pi}
\\ &\quad +\sum_{j=1}^{M}\int_{0}^{2\pi}\max_{K}\sum_{l \in K}\log \dfrac {\|f(c_j+\dfrac{1}{r}e^{i\theta})\|}
{|(a_l,f)(c_j+\dfrac{1}{r}e^{i\theta})|} \dfrac{d\theta}{2\pi}+O(1).
\end{align*}
\end{lemma}
\begin{proof} Let $a_{l} = (a^{(l)}_0,\dots,a^{(l)}_n)$  be the associated vector of $H_{l}$, $1\leqslant l \leqslant q,$ and let 
$\mathcal T$ be the set of all injective maps $\mu:\{ 0,1,\dots,n\} \longrightarrow \{1,\dots ,q\}$. By hypothesis 
$ H_{1},\dots, H_{q}$ are in general position for any $\mu \in \mathcal T,$ then the vectors $a_{\mu(0)},\dots,a_{\mu(n)}$ are 
linearly independent.  

Let $\mu \in T$, we have
\begin{align}\label{ct3.10}
(f,a_{\mu(t)})=a^{\mu(t)}_0f_{0}+\dots +a^{\mu(t)}_nf_{n},\  t=0,1,\dots,n.
\end{align}
Solve the system of linear equations (\ref{ct3.10}), we get
$$ f_{t}=b_{0}^{\mu(t)}(a_{0}^{\mu(t)},f)+\dots +b_{n}^{\mu(t)}(a_{n}^{\mu(t)},f), \ t=0,1,\dots,n,$$
where$ \biggl(b_{j}^{\mu(t)}\biggl)_{t,j=0}^{n}$ is the inverse matrix of $ \biggl(a_{j}^{\mu(t)}\biggl)_{t,j=0}^{n}.$ So there is a constant $C_\mu$ satisfying
$$\| f(z)\| \leqslant C_\mu \max_{0\leqslant t \leqslant n}| (a_{\mu(t)},f)(z)|.$$
Set $C = \max\limits_{\mu \in \mathcal T} C_\mu$. Then for any $\mu \in \mathcal T$, we have
$$\| f(z)\| \leqslant C \max_{0\leqslant t \leqslant n}| (a_{\mu(t)},f)(z)|.$$
For any $z\in\Omega$, there exists the mapping $\mu\in \mathcal  T$ such that
$$ 0 <  |(a_{\mu(0)},f)(z)| \leqslant  |(a_{\mu(1)},f)(z) | \leqslant \dots .\leqslant | (a_{\mu(n)},f)(z)| \leqslant
 | (a_{l},f)(z)| ,$$ 
 for $l \notin \{\mu(0),\dots,\mu(n)\}$. Hence
$$\prod_{l=1}^{q}\dfrac {\| f(z)\|}{| (a_{l},f)(z)|} \leqslant  C ^{q-n-1}\max _{\mu\in \mathcal  T}
\prod_{t=0}^{n}\dfrac {\| f(z)\|}{| (a_{\mu(t)},f)(z)|}.$$
We have
\begin{align*}
\sum_{l=1}^{q}&m_{f}(r,H_{l}) \\
&=\sum_{l=1}^{q}\dfrac{1}{2\pi}\int_{0}^{2\pi} \log \dfrac {\|f(re^{i\theta})\|}{|( a_{l},f)(re^{i\theta})|}\,\dfrac{d\theta}{2\pi}
+\sum_{l=1}^{q}\dfrac{1}{2\pi}\sum_{j=1}^{M}\int_{0}^{2\pi} \log \dfrac {\|f(c_j+\dfrac{1}{r}e^{i\theta})\|}
{|( a_{l},f)(c_j+\dfrac{1}{r}e^{i\theta})|}\,\dfrac{d\theta}{2\pi}\\
&=\dfrac{1}{2\pi}\int_{0}^{2\pi} \log \prod_{l=1}^{q}\dfrac {\|f(re^{i\theta})\|}{|( a_{l},f)(re^{i\theta})|}\,\dfrac{d\theta}{2\pi}
+\dfrac{1}{2\pi}\sum_{j=1}^{M}\int_{0}^{2\pi} \log \prod_{l=1}^{q}\dfrac {\|f(c_j+\dfrac{1}{r}e^{i\theta})\|}
{|( a_{l},f)(c_j+\dfrac{1}{r}e^{i\theta})|}\,\dfrac{d\theta}{2\pi}.
\end{align*}
This implies
\begin{align*}
\sum_{l=1}^{q}&m_{f}(r,H_{l}) \\
&\leqslant \dfrac{1}{2\pi}\int_{0}^{2\pi} \log \max _{\mu\in \mathcal  T}\prod_{t=0}^{n}\dfrac {\|f(re^{i\theta})\|}
{|( a_{\mu(t)},f)(re^{i\theta})|}\,\dfrac{d\theta}{2\pi}\\
&\hspace{2cm} + \dfrac{1}{2\pi}\sum_{j=1}^{M}\int_{0}^{2\pi} \log \max _{\mu\in \mathcal  T}\prod_{t=0}^{n}
\dfrac {\|f(c_j+\dfrac{1}{r}e^{i\theta})\|}{|( a_{\mu(t)},f)(c_j+\dfrac{1}{r}e^{i\theta})|}\,\dfrac{d\theta}{2\pi}+O(1) \\
&=\dfrac{1}{2\pi}\int_{0}^{2\pi} \max _{\mu\in T} \log\prod_{t=0}^{n}\dfrac {\|f(re^{i\theta})\|}{|( a_{\mu(t)},f)(re^{i\theta})|}\,\dfrac{d\theta}{2\pi}\\
&\hspace{2cm} + \dfrac{1}{2\pi}\sum_{j=1}^{M}\int_{0}^{2\pi} \max _{\mu\in T}\log\prod_{t=0}^{n}
\dfrac {\|f(c_j+\dfrac{1}{r}e^{i\theta})\|}{|( a_{\mu(t)},f)(c_j+\dfrac{1}{r}e^{i\theta})|}\,\dfrac{d\theta}{2\pi}+O(1) \\
&=\dfrac{1}{2\pi}\int_{0}^{2\pi} \max _{\mu\in T} \sum_{t=0}^{n}\log\dfrac {\|f(re^{i\theta})\|}{|( a_{\mu(t)},f)(re^{i\theta})|}\,\dfrac{d\theta}{2\pi}\\
&\hspace{2cm} + \dfrac{1}{2\pi}\sum_{j=1}^{M}\int_{0}^{2\pi}  \max _{\mu\in T}\sum_{t=0}^{n}
\log\dfrac {\|f(c_j+\dfrac{1}{r}e^{i\theta})\|}{|( a_{\mu(t)},f)(c_j+\dfrac{1}{r}e^{i\theta})|}\,\dfrac{d\theta}{2\pi}+O(1).
\end{align*}
So we obtain
\begin{align*}
\sum_{l=1}^{q}m_{f}(r,H_{l})&\leqslant \int_{0}^{2\pi}\max_{K}\sum_{l \in K}\log \dfrac {\|f(re^{i\theta})\|}
{|(a_l,f)(re^{i\theta})|}\,\dfrac{d\theta}{2\pi}\\
&\hspace{2cm}+ \sum_{j=1}^{M}\int_{0}^{2\pi}\max_{K}\sum_{l \in K}\log \dfrac {\|f(c_j+\dfrac{1}{r}e^{i\theta})\|}
{|( a_l,f)(c_j+\dfrac{1}{r}e^{i\theta})|}\,\dfrac{d\theta}{2\pi}+O(1).
\end{align*}
This is conclusion of the Lemma \ref{lm23}.
\end{proof}

\begin{proof}[Proof of Theorem \ref{Th2}]
By Lemma \ref{lm31} and Lemma \ref{lm32}, we obtain
\begin{align}\label{ct3.7}
\| \quad \sum_{l=1}^{q}m_{f}(r,H_{l}) &\leqslant \int_{0}^{2\pi}\max_{K}\sum_{l \in K}\log \dfrac {\|f(re^{i\theta})\|}
{|(a_l,f)(re^{i\theta})|} \dfrac{d\theta}{2\pi}\\ 
&\qquad +\sum_{j=1}^{M}\int_{0}^{2\pi}\max_{K}\sum_{l \in K}\log \dfrac {\|f(c_j+\dfrac{1}{r}e^{i\theta})\|}
{|(a_l,f)(c_j+\dfrac{1}{r}e^{i\theta})|} \dfrac{d\theta}{2\pi}+O(1) \notag\\
&\leqslant (n+1)T_{f}(r) - N_{W}(r,0) +O(\log r +\log T_{f}(r)).\notag
\end{align}
By Theorem \ref{Th1}, we get that
$$T_{f}(r)=N_{f}(r,H_{j})+m_{f}(r,H_{j})+O(1)$$
for any $j \in\{1,\dots,q\}$. So from (\ref{ct3.7}), we have
\begin{align}\label{ct3.8}
\| \quad (q-n-1)T_{f}(r)\leqslant \sum_{l=1}^{q}N_{f}(r,H_{l})- N_{W}(r,0)+O(\log r + \log T_{f}(r)).
\end{align}
For  $z_0 \in \overline \Omega_{r}$, we may assume that $(a_{l},f)$ vanishes at $z_0$ for $ 1\leqslant l \leqslant q_{1}$, 
$(a_{l},f)$ does not vanish at $z_0$ for $ l> q_{1}$. Hence, there exists a integer $k_{l}$ and nowhere vanishing holomorphic 
function $ g_{l}$ in neighborhood $U$ of $z$ such that
$$(a_{l},f)(z)=(z-z_0)^{k_{l}}g_{l}(z), \text{ for } l=1,\dots ,q,$$
here $k_{l}=0$ for $q_{1} <l \leqslant q$. We may assume that $ k_{l}\geqslant n$ for  $1\leqslant l \leqslant q_{0}$, and 
$1\leqslant k_{l}< n$ for $q_{0}< l \leqslant q_{1}.$ By property of the Wronskian, we have
$$W(f)=C.W((a_{\mu(0)},f),\dots .,(a_{\mu(n)},f))=\prod_{l=1}^{q_{0}}(z-z_0)^{k_{l}-n}h(z),$$
where $h(z)$ is holomorphic function on $U$. Then $W(f)$ is vanishes at $z_0$ with order at least
 $$\sum\limits_{l=1}^{q_{0}} (k_{l}-n)=\sum\limits_{l=1}^{q_{0}} k_{l}-q_{0}n.$$
By the definition of $ N_{f}(r,H),\  N_{W}(r,0)$ and $ N_{f}^{n}(r,H)$, we have
\begin{align*}
\sum_{l=1}^{q}N_{f}(r,H_{j})-N_{W}(r,0) &= \sum_{l=1}^{q}N_{f}(r,H_{l})-N_{0}(r,\dfrac{1}{W})+O(1) \\
&\leqslant \sum_{l=1}^{q}N^{n}_{f}(r,H_{l}) +O(1).
\end{align*}
So from (\ref{ct3.8}), we have
$$\| \quad (q-n-1)T_{f}(r)\leqslant \sum_{l=1}^{q}N^{n}_{f}(r,H_{l})+O(\log r+\log T_{f}(r)).$$
The proof of Theorem \ref{Th2} is completed.
\end{proof}

\begin{proof}[Proof of Theorem \ref{Th3}]

Let $D_1, \dots, D_q$ be the hypersurfaces in $\PP^{N}(\C)$ which are located in general position on $V.$ Let $Q_l, 1\le l\le q$ be the homogeneous
polynomials in $\C[x_0, \dots, x_N]$ of degree $d_l$ defining on $D_l.$ We can replace $Q_l$ by $Q_l^{d/d_l},$ where $d$ is the
$l.c.m$ of $d_l,$ $l=1, \dots, q,$ we may assume that $Q_1, \dots, Q_q$ have the same degree of $d.$ 

Given $z\in \Omega$ there exists a renumbering $\{i_0,\dots, i_{n}\}$ of the indices $\{1,\dots , q \}$
such that
\begin{align}\label{cta1}
0<|Q_{i_0}\circ (f (z ))| \le  |Q_{i_2}\circ (f(z))| \le  \dots \le |Q_{i_{n}}\circ (f(z))|\le \min_{l\not \in\{i_0, \dots, i_n\}}
|Q_{l}\circ (f(z))|.
\end{align}
Suppose that $P_1, \dots, P_s$  is a base of algebraic variety $V.$ From the hypothesis,
$D_1, \dots, D_q$ are hypersurfaces in $\PP^{N}(\C)$ which are located in general position on $V,$ we have for every subset
$\{i_0, \dots, i_n\} \subset \{1, \dots, q\},$
$$ V\cap \text{Supp}D_{i_{0}} \cap \dots \cap \text{Supp}D_{i_n}= \emptyset.$$
This implies 
$$ P_1\cap\dots \cap P_s\cap \text{Supp}D_{i_{0}} \cap \dots \cap \text{Supp}D_{i_n}= \emptyset.$$
Thus by Hilbert’s Nullstellensatz \cite{VW}, for
any integer k, $0 \le k \le N,$ there is an integer $m_k > \{d, \max_{t=1}^{s}\{\deg P_t\}\}$ such that
$$ x_k^{m_k}=\sum_{j=0}^{n}b_{k_j}(x_0, \dots, x_N) Q_{i_j}(x_0, \dots, x_N)
+\sum_{t=1}^{s}b_t(x_0, \dots, x_N)P_t(x_0, \dots, x_N),$$
where $b_{k_j}$ are homogeneous forms with coefficients in $\mathbb C$ of degree $m_k-d$ and $b_t$ 
are homogeneous forms with coefficients in $\mathbb C$ of degree $m_k-\deg P_t,$ $t=1, \dots, s.$ So from 
$f:\Omega \to V,$ we have
$$ \sum_{t=1}^{s}b_t(f_0(z), \dots, f_N(z))P_t(f_0(z), \dots, f_N(z))=0.$$
This implies
$$ 
|f_k(z )|^{m_k}\le c_1||f (z )||^{m_k-d}\max\{|Q_{i_0}\circ (f (z ))|, \dots, |Q_{i_n}\circ (f (z ))|\},$$
where $c_1$ is a positive constant depends only on the coefficients of $b_{k_j}, 0\le j\le n, 0\le k\le N,$ thus depends only on the coefficients of $Q_l, 1\le l\le q.$ Therefore,
\begin{align}\label{cta2}
||f(z )||^{d}\le c_1\max\{|Q_{i_0}\circ (f (z ))|, \dots, |Q_{i_n}\circ (f (z ))|\}.
\end{align}
By (\ref{cta1}) and (\ref{cta2}), we get
$$ \prod_{l=1}^{q}\dfrac{||f(z)||^{d}||Q_l||}{|Q_l(f(z))|}\le C
\prod_{k=0}^{n}\dfrac{||f(z)||^{d}||Q_{i_k}||}{|Q_{i_k}(f(z))|},$$
where $C=c_1^{q-n-1}\prod_{l\not\in \{i_0, \dots, i_n\}}||Q_l||$ and $||Q_l||$ is the maximum of
the absolute values of the coefficients of $Q_l.$ Thus, we have
\begin{align*}
\sum_{l=1}^{q}m_f(r, D_l)&=\int_{0}^{2\pi}\sum_{l=1}^{q}\log \dfrac{||f(re^{i\theta})||^{d}}{|Q(f)(re^{i\theta})|}
\dfrac{d\theta}{2\pi}+\sum_{j=1}^{M}\int_{0}^{2\pi}\sum_{l=1}^{q}\log \dfrac{||f(c_j+\dfrac{1}{r}e^{i\theta})||^{d}}
{|Q(f)(c_j+\dfrac{1}{r}e^{i\theta})|}\dfrac{d\theta}{2\pi}\\
&=\int_{0}^{2\pi}\log \prod_{l=1}^{q}\dfrac{||f(re^{i\theta})||^{d}}{|Q(f)(re^{i\theta})|}
\dfrac{d\theta}{2\pi}+\sum_{j=1}^{M}\int_{0}^{2\pi}\log \prod_{l=1}^{q}\dfrac{||f(c_j+\dfrac{1}{r}e^{i\theta})||^{d}}
{|Q(f)(c_j+\dfrac{1}{r}e^{i\theta})|}\dfrac{d\theta}{2\pi}.
\end{align*}
Hence, we get
\begin{align}\label{ct42}
\sum_{l=1}^{q}m_f(r, D_l)&\le \int_{0}^{2\pi}\max_{\{i_0, \dots, i_n\}}\Big\{\log \prod_{k=0}^{n}
\dfrac{||f(re^{i\theta})||^{d}}{|Q_{i_k}(f)(re^{i\theta})|}\Big\}\dfrac{d\theta}{2\pi}\notag\\
&+\sum_{j=1}^{M}\int_{0}^{2\pi}\max_{\{i_0, \dots, i_n\}}\Big\{\log \prod_{k=0}^{n}
\dfrac{||f(c_j+\dfrac{1}{r}e^{i\theta})||^{d}}{|Q_{i_k}(f)(c_j+\dfrac{1}{r}e^{i\theta})|}\Big\}\dfrac{d\theta}{2\pi}+O(1)
\end{align}
By argument as M. Ru \cite{Ru3}, we consider the map 
$$\psi: x\in V \mapsto [Q_1(x):\dots : Q_q(x)] \in \PP^{q-1}(\C).$$
Put $Y=\psi(V).$ The hypothesis {\it in general position} implies that $\psi$ is a finite morphimsm on $V$ and $Y$ is a complex
projective subvarieties of $\PP^{q-1}(\C)$ and $\dim Y=n,$ $\deg Y:=\Delta\le d^n \deg V.$ For each 
$a=(a_1, \dots, a_q)\in \Z_{\ge 0}^{q},$ denote by $y^{a}=y_1^{a_1}\dots y_q^{a_q}.$ Let $m$ be a positive integer, we consider
the vector space $V_m=\C[y_1, \dots, y_q]_m/ (I_Y)_m,$
where $I_Y$ is the prime ideal which is definied algebraic variety $Y,$ $(I_Y)_m:=\C[y_1, \dots, y_q]_m\cap I_Y.$ Fix a basis
$\{\phi_0, \dots, \phi_{n_m}\}$ of $V_m$, where $n_m+1=H_Y(m)=\dim V_m.$ Set,
$$ F=[\phi_0(\psi\circ f):\dots:\phi_{n_m}(\psi\circ f)]: \Omega \to \PP^{n_m}(\C).$$
Note that, $f$ is algebraically  non-degenerate, then $F$ is also. For any $c\in \R_{\ge 0}^{q}$, the Hilbert function of $Y$ with
respect to the weight $c$ is definied by
$$ S_Y(m, c)=\max \sum_{i=1}^{H_Y(m)}a_i.c,$$
where the maximum is taken over all sets of monomials $y^{a_1}, \dots, y^{a_{H_Y(m)}}$ whose 
$y^{a_1}+(I_Y)_m, \dots, y^{a_{H_Y(m)}}+(I_Y)_m$ is a basis of $\C[y_1, \dots, y_q]_m/(I_Y)_m.$ For every $z\in \Omega,$ denote $c_z=(c_{1, z}, \dots, c_{q, z})$, where $c_{l, z}=\log \dfrac{||f(z)||^{d}||Q_l||}{|Q_l(f(z))|},$ $l=1, \dots, q.$ We see that 
$c_z\in \R_{\ge 0}^{q},$ for all $z\in \Omega.$ There exists a subset  $I_z\subset \{0, \dots, q_m\}, q_m=C_{q+m-1}^{m}-1,
  |I_z|=n_m+1=H_Y(m)$ which $\{y^{a_i}: i\in I_z\}$ is a basis of $\C[y_1, \dots, y_q]_m/(I_Y)_m$ 
(residue classes modulo $(I_Y)_m$) and
$$ S_Y(m, c_z)=\max \sum_{i=1}^{H_Y(m)}a_i.c_z.$$
  From two basis $\{y^{a_i}: i\in I_z\}$ and $\{\phi_0, \dots, \phi_{n_m}\},$ there exist the 
forms independent linearly $\{L_{l, z}, l\in I_z\}$ such that 
$$y^{a_l}=L_{l, z}(\phi_0, \dots, \phi_{n_m}).$$
We denote $J$ by the set of indices of the linear forms $L_{l, z}.$  We see
$$ \log \prod_{i\in J}\dfrac{1}{|L_i(F)(z)|}=\log \prod_{i\in J}\dfrac{1}{|Q_1(f)(z)|^{a_{i, 1}}\dots |Q_q(f)(z)|^{a_{i, q}}}.$$
This implies
\begin{align}\label{ct43}
\max_{J} \log \prod_{i\in J}\dfrac{||F(z)||}{|L_i(F)(z)|}&\ge S_Y(m, c_z)-dmH_Y(m) \log ||f(z)||\notag\\
&+(n_m+1)\log ||F(z)||.
\end{align}
By Lemma \ref{lm23}, we have
\begin{align}\label{ct44}
S_Y(m, c_z)\ge \dfrac{mH_Y(m)}{(n+1)\Delta}e_Y(c_z)-H_Y(m)(2n+1)\Delta .\max_{1\le i\le q}c_{i, z}.
\end{align}
From Lemma \ref{lm24} and $D_1, \dots, D_q$ are in general position on $V$, for any
 $\{i_0, \dots, i_n\}\subset \{1, \dots, q\}$, we have
\begin{align}\label{ct45}
E_Y(c_z)\ge (c_{i_0, z}+\dots+c_{i_n, z})\Delta.
\end{align}
Using the definition of $c_z$, we obtain
\begin{align}\label{ct46}
c_{i_0, z}+\dots+c_{i_n, z}=\log\Big(\dfrac{||f(z)||^{d}||Q_{i_0}||}{|Q_{i_0}(f)(z)|}\dots 
 \dfrac{||f(z)||^{d}||Q_{i_n}||}{|Q_{i_n}(f)(z)|}\Big).
\end{align}
From (\ref{ct43}) to (\ref{ct46}), we have
\begin{align}\label{ct47}
\log\Big(&\dfrac{||f(z)||^{d}||Q_{i_0}||}{|Q_{i_0}(f)(z)|}\dots  \dfrac{||f(z)||^{d}||Q_{i_n}||}{|Q_{i_n}(f)(z)|}\Big)\notag\\
&\le \dfrac{(n+1)}{mH_Y(m)}\Big(\max_{J} \log \prod_{l\in J}\dfrac{||F(z)||}{|L_l(F)(z)|}-(n_m+1)\log ||F(z)||\Big)\notag\\
&\quad +d(n+1)\log ||f(z)|| +\dfrac{(2n+1)(n+1)\Delta}{m}\max_{1\le i\le q}c_{i, z}\notag\\
&=\dfrac{(n+1)}{mH_Y(m)}\Big(\max_{J} \log \prod_{l\in J}\dfrac{||F(z)||}{|L_l(F)(z)|}-(n_m+1)\log ||F(z)||\Big)\notag\\
&\quad +d(n+1)\log ||f(z)|| \notag\\
&\quad +\dfrac{(2n+1)(n+1)\Delta}{m}\Big(\max_{1\le l\le q}\log \dfrac{||f(z)||^{d}||Q_l||}{|Q_l(f)(z)|}\Big).
\end{align}
Take the integration of both sides of (\ref{ct47}), we have
\begin{align*}
\int_{0}^{2\pi} &\max_{\{i_0, \dots, i_n\}}\log \Big(
\dfrac{||f(re^{i\theta})||^{d}||Q_{i_0}||}{|Q_{i_0}(f)(re^{i\theta})|}\dots  
\dfrac{||f(re^{i\theta})||^{d}||Q_{i_n}||}{|Q_{i_n}(f)(re^{i\theta})|} \Big)\dfrac{d\theta}{2\pi}\\
&+\sum_{j=1}^{M}\int_{0}^{2\pi} \max_{\{i_0, \dots, i_n\}}\log \Big(
\dfrac{||f(c_j+\dfrac{1}{r}e^{i\theta})||^{d}||Q_{i_0}||}{|Q_{i_0}(f)(c_j+\dfrac{1}{r}e^{i\theta})|}\dots  
\dfrac{||f(c_j+\dfrac{1}{r}e^{i\theta})||^{d}||Q_{i_n}||}{|Q_{i_n}(f)(c_j+\dfrac{1}{r}e^{i\theta})|} \Big)
\dfrac{d\theta}{2\pi}\\
&\le \dfrac{(n+1)}{mH_Y(m)}\Big(\int_{0}^{2\pi}\max_{J} \log \prod_{l\in J}
\dfrac{||F(re^{i\theta})||}{|L_l(F)(re^{i\theta})|}\dfrac{d\theta}{2\pi}\Big)\\
&+\sum_{j=1}^{M}\dfrac{(n+1)}{mH_Y(m)}\Big(\int_{0}^{2\pi}\max_{J} \log \prod_{l\in J}
\dfrac{||F(c_j+\dfrac{1}{r}e^{i\theta})||}{|L_l(F)(c_j+\dfrac{1}{r}e^{i\theta})|}\dfrac{d\theta}{2\pi}
-(n_m+1)T_F(r)\Big)\\
&+d(n+1)T_f(r)+\dfrac{(2n+1)(n+1)\Delta}{m}\sum_{1\le l\le q}m_f(r, D_l)
\end{align*}
This implies
\begin{align}\label{ct48}
\int_{0}^{2\pi} &\max_{\{i_0, \dots, i_n\}}\log \Big(
\dfrac{||f(re^{i\theta})||^{d}}{|Q_{i_0}(f)(re^{i\theta})|}\dots  
\dfrac{||f(re^{i\theta})||^{d}}{|Q_{i_n}(f)(re^{i\theta})|} \Big)\dfrac{d\theta}{2\pi}\notag\\
&+\sum_{j=1}^{M}\int_{0}^{2\pi} \max_{\{i_0, \dots, i_n\}}\log \Big(
\dfrac{||f(c_j+\dfrac{1}{r}e^{i\theta})||^{d}}{|Q_{i_0}(f)(c_j+\dfrac{1}{r}e^{i\theta})|}\dots  
\dfrac{||f(c_j+\dfrac{1}{r}e^{i\theta})||^{d}}{|Q_{i_n}(f)(c_j+\dfrac{1}{r}e^{i\theta})|} \Big)
\dfrac{d\theta}{2\pi}\notag\\
&\le \dfrac{(n+1)}{mH_Y(m)}\Big(\int_{0}^{2\pi}\max_{J} \log \prod_{l\in J}
\dfrac{||F(re^{i\theta})||}{|L_l(F)(re^{i\theta})|}\dfrac{d\theta}{2\pi}\Big)\notag\\
&+\sum_{j=1}^{M}\dfrac{(n+1)}{mH_Y(m)}\Big(\int_{0}^{2\pi}\max_{J} \log \prod_{l\in J}
\dfrac{||F(c_j+\dfrac{1}{r}e^{i\theta})||}{|L_l(F)(c_j+\dfrac{1}{r}e^{i\theta})|}\dfrac{d\theta}{2\pi}\Big)\notag\\
&-\dfrac{(n+1)}{mH_Y(m)}(n_m+1)T_F(r)+d(n+1)T_f(r)\notag\\
&+\dfrac{(2n+1)(n+1)\Delta}{m}\sum_{1\le l\le q}m_f(r, D_l)+O(1).
\end{align}
Next, apply to Lemma \ref{lm31} for $F$ and collection of hyperplanes $L_l, l\in J,$ for every $\varepsilon>0$ and $m$ is large enough, 
we obtain
\begin{align}\label{ct49}
\|  \quad&\dfrac{(n+1)}{mH_Y(m)}\Big(\int_{0}^{2\pi}\max_{J} \log \prod_{l\in J}
\dfrac{||F(re^{i\theta})||}{|L_l(F)(re^{i\theta})|}\dfrac{d\theta}{2\pi}\Big)\notag\\
&+\sum_{j=1}^{M}\int_{0}^{2\pi}\max_{J} \log \prod_{l\in J}
\dfrac{||F(c_j+\dfrac{1}{r}e^{i\theta})||}{|L_l(F)(c_j+\dfrac{1}{r}e^{i\theta})|}\dfrac{d\theta}{2\pi}-(n_m+1)T_F(r)\Big)\notag\\
&\le -\dfrac{n+1}{mH_Y(m)}N_W(r, 0)+\dfrac{\varepsilon}{3m}T_F(r)+O(\log r+\log T_F(r)),
\end{align}
 where $N_W(r, 0)$  is denoted by the Wronskian of $F.$  Combining (\ref{ct42}), (\ref{ct48}) and (\ref{ct49}), we get
\begin{align}\label{ct50}
\|  \quad \sum_{l=1}^{q}m_f(r, D_l)&\le -\dfrac{n+1}{mH_Y(m)}N_W(r, 0)+\dfrac{\varepsilon}{3m}T_F(r)+d(n+1)T_f(r)\notag\\
&\quad+\dfrac{(2n+1)(n+1)\Delta}{m}\sum_{1\le l\le q}m_f(r, D_l)\notag\\
&\quad+O(\log r+\log T_{F}(r)).
\end{align}
Using the Theorem \ref{Th1}, we see $T_F(r)\le dmT_f(r)+O(1).$ Thus, (\ref{ct50}) implies
\begin{align}\label{ct51}
\|  \quad \sum_{l=1}^{q}d(q-(n+1)-\varepsilon/3)T_f(r)&\le \sum_{l=1}^{q}N_f(r, D_l) -\dfrac{n+1}{mH_Y(m)}N_W(r, 0)\notag\\
&\quad+\dfrac{(2n+1)(n+1)\Delta}{m}\sum_{1\le l\le q}m_f(r, D_l)\notag\\
&\quad+O(\log r+\log T_f(r)).
\end{align}
By an argument method in \cite{Ru4}, we conclude 
\begin{align}\label{ct52}
 \dfrac{n+1}{mH_Y(m)}\sum_{l=1}^{q}N_f(r, D_l) -N_W(r, 0)&\le  \dfrac{n+1}{mH_Y(m)}\sum_{l=1}^{q}N_f^{n_m}(r, D_l)\notag\\
&+(2n+1)\Delta H_Y(m)\sum_{l=1}^{q}N_f(r, D_l).
\end{align}
Combining (\ref{ct51}) and (\ref{ct52}), we have
\begin{align}\label{ct53a}
\|  \quad d(q-(n+1)-\varepsilon/3)T_f(r)&\le \sum_{l=1}^{q}N_f^{n_m}(r, D_l)
+\dfrac{(2n+1)(n+1)\Delta}{m}\sum_{l=1}^{q}N_f(r, D_l)\notag\\
&\quad +\dfrac{(2n+1)(n+1)\Delta}{m}\sum_{l=1}^{q}m_f(r, D_l)\notag\\
&\le \sum_{l=1}^{q}N_f^{n_m}(r, D_l)+\dfrac{(2n+1)(n+1)dq\Delta}{m}T_f(r)\notag\\
&\quad+O(\log r+\log T_f(r)).
\end{align}
We choose the $m$ sufficiently large such that 
\begin{align}\label{ct53}
\dfrac{(2n+1)(n+1)\Delta}{m}<\varepsilon/3.
\end{align}
We may choose $m=18n^2\Delta I(\varepsilon^{-1})$ for the inequality (\ref{ct53}), where $I(x):=\min\{k\in \N: k>x\}$ for each positive
constant $x.$  Thus, from (\ref{ct53a}) and (\ref{ct53}), we get the inequality
\begin{align*}
\|  \quad d(q(1-\varepsilon/3)-(n+1)-\varepsilon/3)T_f(r)&\le \sum_{l=1}^{q}N_f^{n_m}(r, D_l)\\
&+O(\log r+\log T_f(r)).
\end{align*}

By property $\deg Y= \Delta \le d^n \deg V,$ where $d=lcm \{d_1, \dots, d_q\}$, $\deg Y =n$ and  $n_m\le \Delta C_{m+n}^{n},$ we 
have
\begin{align*}
n_m&\le \Delta \dfrac{(m+1)(m+2)\dots (m+n)}{n!}\\
&<\Delta \Big(\dfrac{m+n}{n}\Big)^{n}\dfrac{n^n}{n!}\\
&=\Delta \Big(1+\dfrac{m}{n}\Big)^{n}\dfrac{n^n}{n!}.
\end{align*}
For the choice of $m,$ we have
$$ n_m\le \dfrac{n^nd^{n^2+n}(19nI(\varepsilon^{-1}))^n.(\deg V)^{n+1}}{n!}.$$
\end{proof}

\begin{proof}[Proof of Theorem \ref{Th10}]

Let ${\large \text{f}}= (f_0:\dots: f_N)$ be a reduced representation of $f,$ where $f_0, \dots, f_N$ are entire functions on 
$\Omega$ and have no common zeros. We consider the function 
$\phi_i=Q_i\circ {\large \text{f}}=Q_i(f_0, \dots, f_N), 0\le i\le N.$ Let 
 $F=(\phi_0f_0^n:\dots: \phi_Nf_N^n).$ Since the hypersurfaces  $\{\mathcal H_i^nQ_i=0\}, 0\le i\le N,$ are 
located in general position 
in $\mathbb P^{N}(\mathbb C),$ then $F: \Omega \to \mathbb P^N(\mathbb C)$ is a holomorphic curve. Let 
 $\mathfrak H_i, 0\le i\le N,$ be the hypersurface defined by $\{\mathcal H_i^nQ_i=0\}, 0\le i\le N.$
From the hypothesis  $\mathfrak H_0, \dots, \mathfrak H_N$ are in general position, i.e.
$$ \text{supp}\mathfrak H_{0} \cap \dots \cap \text{supp}\mathfrak H_N= \emptyset.$$
Thus by Hilbert's Nullstellensatz \cite{VW}, for
any integer k, $0 \le k \le N,$ there is an integer $m_k > n+d$ such that
$$ x_k^{m_k}=\sum_{i=0}^{N}b_{i}(x_0, \dots, x_N)  \mathcal H_i^n(x_0, \dots, x_N)Q_{i}(x_0, \dots, x_N),$$
where $b_{0}, \dots, b_N$ are homogeneous forms with coefficients in $\mathbb C$ of degree $m_k-(n+d).$  
This implies
$$ 
|f_k(z )|^{m_k}\le c_1||{\large \text{f}}(z )||^{m_k-(n+d)}\max\{|\mathcal H_0^nQ_{0}({\large \text{f}}(z ))|, \dots, |\mathcal H_N^nQ_{N}
({\large \text{f}}(z ))|\},$$
where $c_1$ is a positive constant depending only on the coefficients of $b_{i}, 0\le i\le N, 0\le k\le N,$ thus depending only on the coefficients of $Q_i, 0\le i\le N.$ Therefore,
\begin{align}\label{cta2}
||{\large \text{f}}(z )||^{n+d}\le c_1\max\{|\mathcal H_0^nQ_{0}({\large \text{f}}(z ))|, \dots, |\mathcal H_N^nQ_{N}({\large \text{f}}(z ))|\}.
\end{align}

From (\ref{cta2}) and the First Main Theorem, we have
\begin{align}\label{41}
 T_{F}(r)&\ge (n+d)T_{f}(r)+O(1)\notag\\
&\ge (n+d-(N+1)d)T_{f}(r)+\sum_{i=0}^{N}N_{f}(r, D_i)+O(1)\notag\\
&= (n-Nd) T_{f}(r)+\sum_{i=0}^{N}N_{f}(r, D_i)+O(1).
\end{align}
On the other hand, by applying Theorem \ref{Th2} to $F$, and the hyperplanes 
$$ H_i=\{y_i=0\}, 0\le i\le N,$$
and
$$ H_{N+1}=\{y_0+\dots+y_{N}=0\} $$
yields
\begin{align}\label{42}
\| T_{F}(r)\le \sum_{i=0}^{N+1} N_{F}^{N}(r, H_i)+o( T_{f}(r)).
\end{align}
We have
\begin{align*}
N_{F}^{N}(r, H_i)\le  N_{f}^{N}(r, D_i)+ N^{N}(r, \dfrac{1}{f_i^n})
\end{align*}
for all $i=0, \dots, N,$ where $ N^N(r, \dfrac{1}{g})$ is counting function with level of truncation $N$ of $g.$ Hence
\begin{align}\label{43}
 N_{ F}^{N}(r, H_i)&\le N_{f}^{N}(r, D_i)+N\overline { N}
(r, \dfrac{1}{f_i^n})\notag\\
&\le  N_{f}^{N}(r, D_i)+N T_{f}(r)+O(1)
\end{align}

for all $i=0, \dots, N.$ Also note $N_{F}^{N}(r, H_{N+1})= N_{ f}^{N}(r, D).$ 
By combining (\ref{41}) to  (\ref{43}), we obtain 
\begin{align*}
\|(n-(d+N+1)N) T_{ f}(r)&+\sum_{i=0}^{N}( N_{ f}(r, D_i)-
N_{f}^{N}(r, D_i))\\
&\le 
N_{f}^{N}(r, D)+o( T_{ f}(r)).
\end{align*}
\end{proof}

\begin{proof}[Proof of Theorem \ref{Th6}]
We suppose that $f\not\equiv g,$ then there are two numbers $\alpha, \beta \in  \{0, \dots , N\},$ $ \alpha\ne \beta$ such that
$f_{\alpha}g_{\beta}\not\equiv f_{\beta}g_{\alpha}.$ Assume that $z_0 \in \Omega$ is a zero of $P(f),$  from condition 
$f(z)=g(z)$ when $z\in f^{-1}(D) \cup g^{-1}(D),$ we get $f(z_0)=g(z_0).$ This implies $z_0$ is a zero of $\dfrac{f_{\alpha}}{f_{\beta}}-\dfrac{g_{\alpha}}{g_{\beta}}.$ Therefore, we have
\begin{align*}
 N_f^{N}(r, D) \le NN_f^{1}(r, D)&\le N N_{ \dfrac{f_{\alpha}}{f_{\beta}}-\dfrac{g_{\alpha}}{g_{\beta}}}(r,0 )\\
&\le N(T_f(r)+T_g(r))+O(1).
\end{align*}
Apply to Theorem \ref{Th10}, we obtain
\begin{align}\label{51}
  \|(n-(d+N+1)N)T_f(r)\le N(T_f(r)+T_g(r))+o(T_f(r)).
\end{align}
Similarly, we have
\begin{align}\label{52}
  \|(n-(d+N+1)N)T_g(r)\le N(T_f(r)+T_g(r))+o(T_g(r)).
\end{align}
Combining (\ref{51}) and (\ref{52}), we get
$$ \|(n-(d+N+3)N)(T_f(r)+T_g(r)) \le o(T_f(r))+o(T_g(r)).$$
This is a contradiction with $n>(d+N+3)N.$ Hence $f\equiv g.$
\end{proof}

\end{document}